\documentclass[a4paper,11pt]{article}
\usepackage{amsmath}
\usepackage{amsthm}
\usepackage{amsfonts}
\usepackage{amssymb}
\title{Minimal weight in union-closed families}
\author{Victor Falgas--Ravry\footnote{School of Mathematical Sciences, Queen Mary,
University of London, London E1 4NS, England }}

\begin{document}
\maketitle
\renewcommand{\baselinestretch}{1.1}

\newtheorem{theorem}{Theorem}
\newtheorem{lemma}[theorem]{Lemma}
\newtheorem{corollary}[theorem]{Corollary}
\newtheorem*{quotedresult}{Theorem}
\newtheorem{conjecture}{Conjecture}
\newtheorem*{conjecture*}{Conjecture}
\newtheorem*{definition}{Definition}
\newtheorem*{claim}{Claim}
\newtheorem*{remark}{Remark}
\begin{abstract}

Let $\Omega$ be a finite set and let $\mathcal{S} \subseteq \mathcal{P}(\Omega)$ be a set system on $\Omega$. For $x\in \Omega$, we denote by $d_{\mathcal{S}}(x)$ the number of members of $\mathcal{S}$ containing $x$. A long-standing conjecture of Frankl \cite{frankl} states that if $\mathcal{S}$ is union-closed then there is some $x\in \Omega$ with $d_{\mathcal{S}}(x)\geq \frac{1}{2}|\mathcal{S}|$.

We consider a related question. Define the \emph{weight} of a family $\mathcal{S}$ to be $w(\mathcal{S}) := \sum_{A \in \mathcal{S}} |A|$. Suppose $\mathcal{S}$ is union-closed. How small can $w(\mathcal{S})$ be? Reimer \cite{reimer} showed 
\[w(\mathcal{S}) \geq \frac{1}{2} |\mathcal{S}| \log_2  |\mathcal{S}|,\]
and that this inequality is tight.
In this paper we show how Reimer's bound may be improved if we have some additional information about the domain $\Omega$ of $\mathcal{S}$: if $\mathcal{S}$ separates the points of its domain, then 
\[w(\mathcal{S})\geq \binom{|\Omega|}{2}.\] 
This is stronger than Reimer's Theorem when $\Omega > \sqrt{|\mathcal{S}|\log_2 |\mathcal{S}|}$. In addition we construct a family of examples showing the combined bound on $w(\mathcal{S})$ is tight except in the region $|\Omega|=\Theta (\sqrt{|\mathcal{S}|\log_2 |\mathcal{S}|})$, where it may be off by a multiplicative factor of $2$.

Our proof also gives a lower bound on the average degree: if $\mathcal{S}$ is a point-separating union-closed family on $\Omega$, then
\[ \frac{1}{|\Omega|} \sum_{x \in \Omega} d_{\mathcal{S}}(x) \geq \frac{1}{2} \sqrt{|\mathcal{S}| \log_2 |\mathcal{S}|}+ O(1),\]
and this is best possible except for a multiplicative factor of $2$.
\end{abstract}

\section{Introduction}
Let $\Omega$ be a finite set. We may identify $X \subseteq \Omega$ with its characteristic function 
and consider a collection of subsets of $\Omega$ as a family of functions from $\Omega$ into $\{0,1\}$. For such a family $\mathcal{S}\subseteq \mathcal{P}(\Omega)$, we refer to $\Omega =\Omega(S)$ as the \emph{domain} of $\mathcal{S}$. Note that the domain of a set system $\mathcal{S}$ is not uniquely determined by knowledge $\mathcal{S}$. 
Therefore when we speak of `a set system $\mathcal{S}$', we shall in fact mean `a pair $(\mathcal{S}, \Omega)$, where $\mathcal{S} \subseteq \mathcal{P}(\Omega)$' so that the domain of $\mathcal{S}$ is implicitly specified.

We also let $V(S):= \bigcup_{A\in \mathcal{S}} A$ be the set of all elements $x\in\Omega$ which appear as a member of at least one set $A\in\mathcal{S}$. For $x\in \Omega$ we denote by $d_{\mathcal{S}}(x)$ the number of members of $\mathcal{S}$ containing $x$. We call $d_{\mathcal{S}}(x)$ the \emph{degree} of $x$ in $\mathcal{S}$.

A set system $\mathcal{S}$ is \emph{union-closed} if it is closed under pairwise unions. This is essentially the same as being closed under arbitrary unions except that we do not require $\mathcal{S}$ to contain the empty set. In 1979, Frankl~\cite{frankl} made a simple-sounding conjecture on the maximal degree in a union-closed family. This remains open and has become known as the Union-closed sets conjecture:

\begin{conjecture}[Union-closed sets conjecture]\label{ucconj}
Let $\mathcal{S}$ be a set system on some finite set $\Omega$. Then there is an element $x\in \Omega$ which is contained in at least half of the members of $\mathcal{S}$.
\end{conjecture}
(An equivalent lattice-theoretic version also exists. See for example Abe and Nakano, Poonen or Stanley \cite{abenakano, poonen, stanley}.)

Very little progress has been made on Conjecture \ref{ucconj}. A simple argument due to Knill~\cite{knill} establishes that for any union-closed family $\mathcal{S}$ with $|\mathcal{S}|=m$, there always exists some $x$ contained in at least $\frac{m}{\log_2 m}$ members of $\mathcal{S}$. W\'ojcik \cite{wojcik1} improved this by a multiplicative constant. The conjecture is also known to hold if $|\mathcal{S}|<40$ (see \cite{lofaro, roberts}) or $|V(\mathcal{S})|<11$ (see \cite{morris, bosnmarkov}), if $|\mathcal{S}|>\frac{5}{8} \times 2^{|V(\mathcal{S})|}$ (see \cite{czedli, czedlietal1, czedlietal2}), or if $\mathcal{S}$ contains some very specific collections of small sets (see \cite{morris, bosnmarkov}).

In a different direction, Reimer~\cite{reimer} found a beautiful shifting argument to obtain a sharp lower bound on the average set size of $\mathcal{S}$ as a function of $|\mathcal{S}|$. We state his result here.

\begin{quotedresult}[Reimer's Average Set Size Theorem]\label{reimerweight} 
Let $\mathcal{S}$ be a union-closed family. Then 
\begin{displaymath}
\frac{1}{|\mathcal{S}|}\sum_{A \in \mathcal{S}} |A| \geq \frac{\log_2 |\mathcal{S}|}{2}
\end{displaymath}
with equality if and only if $\mathcal{S}$ is a powerset.
\end{quotedresult}

Define the \emph{weight} of a family $\mathcal{S}$ to be
\begin{align*}
w(\mathcal{S}) &:=  \sum_{A \in \mathcal{S}} |A| \\
&= \sum_{x \in \Omega} d_{\mathcal{S}}(x).
\end{align*}

We shall think of Reimer's Theorem as a lower bound for the smallest possible weight of a union-closed family of a given size. Let $\mathcal{S}$ be a union-closed family. In this form, Reimer's Theorem states that
\begin{align*}
w(\mathcal{S}) &\geq \frac{|\mathcal{S}|\log_2 |\mathcal{S}|}{2}
\end{align*}
with equality if and only if $\mathcal{S}$ is a powerset. The purpose of this paper is to show how we may improve this inequality if we have some additional information about $\Omega(\mathcal{S})$. As a corollary, we also give asymptotically tight (up to a constant) lower bounds on the average degree over $\Omega$, $\frac{1}{|\Omega|}\sum_{x\in \Omega} d_{\mathcal{S}}(x)$.

As we remarked earlier, $\Omega(\mathcal{S})$ is not uniquely specified by $\mathcal{S}$. For example, $\Omega(\mathcal{S})$ could contain many elements which do not appear in $\mathcal{S}$. This would bring the average degree in $\Omega$ arbitrarily close to $0$. Restricting our attention to $V(\mathcal{S})$ does not entirely resolve this problem: pick $x\in V(\mathcal{S})$. Replacing every instance of $x$ in a member of $\mathcal{S}$ by a set $x_1,x_2, \ldots x_M$ for some arbitrarily large $M$ gives us a new union-closed family $\mathcal{S}'$ with the same structure as $\mathcal{S}$ but with average degree over $V(\mathcal{S}')$ arbitrarily close to $d_{\mathcal{S}}(x)$.

Thus to say anything interesting about average degree, we need to impose a restriction on $\mathcal{S}$ and its domain. In particular we want to make sure that no element of $\Omega(\mathcal{S})$ is `cloned' many times over. We make therefore the following natural definition.

\begin{definition}
A family $\mathcal{S}$ \emph{separates} a pair $(i,j)$ of elements of $\Omega(\mathcal{S})$ if there exists $A \in \mathcal{S}$ such that $A$ contains exactly one of $i$ and $j$. $\mathcal{S}$ is \emph{separating} if it separates every pair of distinct elements of $\Omega(\mathcal{S})$. If $|\Omega(S)|=n$ and $\mathcal{S}$ is separating, we say that $\mathcal{S}$ is $n$-separating. 
\end{definition} 

Recalling our identification of sets with their characteristic functions, $\mathcal{S}$ is separating if and only if it separates the points of $\Omega(\mathcal{S})$ as a family of functions $\Omega \rightarrow\{0,1\}$.

Trivially, a family $\mathcal{S}$ of size $|\mathcal{S}|=m$ can be at most $2^m$-separating. In Section 2, we make use of certain heredity properties of union-closed families to prove that if in addition $\mathcal{S}$ is union-closed it can be at most $(m+1)$-separating. The main result of that section, Theorem~\ref{sepweight}, establishes that for any $n$ there is a unique (up to relabelling of vertices) $n$-separating union-closed family of minimal weight.

In the third section, we use Theorem~\ref{sepweight} together with Reimer's Theorem to obtain lower bounds on the weight of $n$-separating union-closed families of size $m$ for every realisable pair $(m,n)$. 

We construct families of examples showing these bounds are sharp up to a multiplicative factor of $2+O\left(\frac{1}{\log_2 m}\right)$.

In the final section we consider a generalisation of our original problem. We define the \emph{$l$-fold weight} of a family $\mathcal{S}$ to be
\[w_l(\mathcal{S}):= \sum_{A \in \mathcal{S}} \binom{|A|}{l}.\] 
The $0$-fold weight of $\mathcal{S}$ is just the size of $\mathcal{S}$, while the $1$-fold weight is the weight $w(\mathcal{S})$ we introduced earlier. Similarly to the $l=1$ case, we can bound $w_l$ below for $l\geq 2$ when $\mathcal{S}$ is separating using a combination of Reimer's Theorem and Theorem~\ref{sepweight} together with some elementary arguments. Again we provide constructions showing our bounds are the best possible up to a multiplicative factor of $2+O\left(1/\log_2 m\right)$. As instant corollaries to our results in sections~3 and~4, we have for any $l\geq 1$ sharp (up to a multiplicative constant) lower bounds on the expected number of sets in $\mathcal{S}$ containing a randomly selected $l$-tuple from $\Omega(\mathcal{S})$. These results are related to a generalisation of the union-closed sets conjecture.

\section{Separation}

In this section we use our definition of \emph{separation} to prove some results about separating union-closed families. We begin with an item of notation. Let $\mathcal{S}$ be a family with domain $\Omega$. Given $X \subseteq\Omega$, we will denote by $\mathcal{S}[X]$ the family \emph{induced} by $X$ on $\mathcal{S}$,
\[\mathcal{S}[X]:=\left\{A\setminus X |A\supseteq  X, A\in \mathcal{S} \right\}.\]
We shall consider $\mathcal{S}[X]$ as a family with domain $\Omega(\mathcal{S}) \setminus X$. In a slight abuse of notation we shall usually write $\mathcal{S}[x]$ for $\mathcal{S}[\{x\}]$. Note that $|\mathcal{S}[x]|=d_{\mathcal{S}}(x)$.

Recall that $\mathcal{S}$ \emph{separates} a pair $(i, j)$ of elements of $\Omega(\mathcal{S})$ if there exists $A \in \mathcal{S}$ such that $A$ contains exactly one of $i$ and $j$. $\mathcal{S}$ is said to be \emph{separating} if it separates every pair of distinct elements of $\Omega(\mathcal{S})$. We introduce an equivalence relation $\cong_{\mathcal{S}}$ on its domain $\Omega(\mathcal{S})$ by setting $x \cong_{\mathcal{S}} y$ if $\mathcal{S}$ does not separate $x$ from $y$. Quotienting $\Omega$ by $\cong_{\mathcal{S}}$ in the obvious way, we obtain a reduced family 
\[\mathcal{S}'= \mathcal{S}/\cong_{\mathcal{S}}\]
on a new domain $\Omega'$ consisting of the $\cong_{\mathcal{S}}$ equivalence classes on $\Omega$. It follows from the definition of $\cong_{\mathcal{S}}$ that $\mathcal{S}'$ is separating and uniquely determined by the knowledge of $\mathcal{S}$ and $\Omega$. We shall refer to $\mathcal{S}'$ as the \emph{reduction} of $\mathcal{S}$.

Union-closure is clearly preserved by our quotienting operation. Every union-closed family $S$ may thus be reduced to a unique separating union-closed family in this way. Such separating union-closed families will be the main object we study in this paper. Before proving anything about them, let us give a few examples.

For $n\geq 2$, we define the \emph{staircase} of height $n$ to be the union-closed family 
\[T_n =\left\{ \{n\}, \{n-1, n\}, \{n-2, n-1, n\}, \ldots \{2,3,  \ldots n\}\right\}\]
with domain $\Omega(T_n)=\{1,2,3 \ldots...n \}$. Note that $T_n$ is $n$-separating, has size $n-1$ and that $V(T_n) \neq \Omega(T_n)$, since the element $1$ is not contained in any set of $T_n$. For completeness, we define $T_1$ to be the empty family with domain $\Omega(T_1)=\{1\}$ and size $0$. Recall that $T_n[X]$ is the subfamily of $T_n$ induced by $X$. $T_n$ has the property that $T_n[\{n\}]=T_{n-1} \cup \{\emptyset\}$.

We shall prove that $T_n$ is an $n$-separating union-closed family of least weight.

For $n \geq 2$, the \emph{plateau} of width $n$ is the $n$-separating union-closed family
\[U_n =\left\{\{1,2,\ldots n-1\}, \{1,2, \ldots n-2, n\}, \ldots \{1,3,4 \ldots n\}, \{2,3, \ldots n\}, [n]\right\}.\]
with domain $\Omega(U_n)=[n]$ and size $n+1$. For completeness we let $E_1$ be the family $\{\emptyset, \{1\}\}$ with domain $\{1\}$.
It is easy to see that $U_n$ is the $n$-separating union-closed family of size $n+1$ with maximal weight. It has weight roughly twice that of $T_n$, and the additional property that for every pair $\{i,j\} \subseteq [n]$ there is a set in $U_n$ containing $i$ and not $j$ as well as a set containing $j$ and not $i$.

Finally, or $n\geq 1$, the powerset of $[n]$, $P_n=\mathcal{P}[n]$ is, of course, a $n$-separating union-closed family with domain $\Omega(P_n)=V(P_n)=[n]$. Note that $P_n[\{n\}]=P_{n-1}$, and that $P_n$ is the largest $n$-separating family in every sense of the word, having both the maximum size and the maximum weight possible.

Let us now turn to the main purpose of this section. 

We begin with a trivial lemma.

\begin{lemma} \label{remarkable}
Let $\mathcal{S}$ be a separating family on $\Omega=[n]$ with elements labelled in order of increasing degree. Then if $1\leq i < j\leq n$ there exists $A \in \mathcal{S}$ with $j\in A$, $i\notin A$.  
\end{lemma}
\begin{proof}
Since $\mathcal{S}$ is separating, there is some $A$ in $\mathcal{S}$ containing one but not both of $i$, $j$. But we also know that $d_{\mathcal{S}}(i) \leq d_{\mathcal{S}}(j)$, so at least one such $A$ contains $j$ and not $i$. 
\end{proof}

Repeated applications of Lemma~\ref{remarkable} yield the following:

\begin{lemma}\label{sepstruc}
Let $\mathcal{S}$ be a separating union-closed family with $\Omega(\mathcal{S})=[n]$ and elements of $\Omega$ labelled in order of increasing degree. Then for every $i \in [n-1]$, $\mathcal{S}$ contains a set $A_i = \left([n]\setminus[i]\right) \cup X_i$, where $X_i \subseteq [i-1]$. These $n-1$ sets are distinct.
\end{lemma}
\begin{proof}
Pick $i \in [n-1]$. By Lemma~\ref{remarkable}, for each $j>i$ there exists $B_j \in \mathcal{S}$ containing $j$ and not $i$. Let $A_i= \bigcup_{j>i} B_j$. By union-closure, $A_i \in \mathcal{S}$. $A_i$ is clearly of the form $\{i+1, i+2, \ldots n\}\cup X_i$, where $X_i$ is a subset of $[i-1]$. Moreover if $i<j$ we have $A_i \neq A_j$ since $j \in A_i$, $j\notin A_j$. 
\end{proof}

The main result of this section follows easily.

\begin{theorem}\label{sepweight}

Let $\mathcal{S}$ be a separating union-closed family on $\Omega(\mathcal{S})=[n]$ with elements labelled in order of increasing degree.  Then $d_{\mathcal{S}}(i) \geq i-1$ for all $i \in [n]$. In particular, $|\mathcal{S}|\geq n-1$, and the weight of $\mathcal{S}$ satisfies :
\[w(\mathcal{S}) \geq \binom{n}{2}.\] 
Moreover, $w(\mathcal{S}) =\binom{n}{2}$ if and only if $\mathcal{S}$ is one of $T_n$ or $T_n \cup\{\emptyset\}$, where $T_n$ is the staircase of height $n$ introduced earlier.
\end{theorem}
\begin{proof}
By Lemma~\ref{sepstruc}, $\mathcal{S}$ contains $n-1$ distinct sets $A_1$, $A_2$, $\ldots A_{n-1}$ such that $[n]\setminus[i] \subseteq A_i$. It follows in particular that $|\mathcal{S}| \geq n-1$ and that $d_{\mathcal{S}}(i) \geq i-1$ for all $i \in [n]$. Moreover
\begin{align*}
w(\mathcal{S}) & \geq \sum_{i \in [n-1]}|A_i|\\
&\geq \sum_{i \in [n-1]} (n-i)=\binom{n}{2}\\ 
\end{align*}
with equality if and only if $A_i=[n]\setminus [i]$ for every $i$ and in addition $\mathcal{S}$ contains no nonempty set other than the $A_i$. Thus $w(\mathcal{S})=\binom{n}{2}$ if and only if $\mathcal{S}$ is one of $T_n$ or $T_n \cup\{\emptyset\}$, as claimed.  
\end{proof}

\section{Minimal weight}
In this section we use Reimer's Theorem and Theorem \ref{sepweight} together to obtain a lower bound on the weight of an $n$-separating union-closed family of size $m$. We then give constructions in the entire range of possible $n$, $\log_2 m \leq n \leq m+1$, showing our bounds are asymptotically sharp except in the region $n =\Theta \left( \sqrt{m\log_2 m} \right)$ (where they are differ by a multiplicative factor of at most $2$). As a corollary, we obtain a lower bound on the average degree in a separating union-closed family.

Let $\mathcal{S}$ be an $n$-separating union-closed family with $|\mathcal{S}|=m$. Recall that the \emph{weight} of $\mathcal{S}$, $w(\mathcal{S})$ is 
\[w(\mathcal{S})=\sum_{A \in \mathcal{S}}|A|=\sum_{x\in \Omega(\mathcal{S})} d_{\mathcal{S}}(x).\]
We know from Reimer's Theorem that 
\[w(\mathcal{S}) \geq \frac{m \log_2 m}{2}.\] 
We have another bound for $w(\mathcal{S})$ coming from our separation result, Theorem~\ref{sepweight}:
\[w(\mathcal{S}) \geq \frac{n(n-1)}{2}.\]
If $n\leq \frac{1}{2}\left(1+ \sqrt{1+4m \log_2 m} \right)=\sqrt{m \log_2 m}+O(1)$, the `bound in $m$' from Reimer's Theorem is stronger; if on the other hand $n \geq \frac{1}{2}\left(1+ \sqrt{1+4m \log_2 m} \right)$, the `bound in $n$' from Theorem \ref{sepweight} is sharper.

For the bound in $m$, equality occurs if and only if $\mathcal{S}$ is a powerset, that is if and only $n= \log_2 m$. For the bound in $n$, equality occurs if and only if $\mathcal{S}$ is a staircase (with possibly the empty set added in). This can only occur if $n=m$ or $n=m+1$. Remarkably the combined bound is asymptotically sharp everywhere except in the region $n =\Theta\left(\sqrt{m \log_2 m} \right)$, where it is only asymptotically sharp up to a constant. We shall show this by constructing intermediate families between powersets and staircases. Roughly speaking these intermediary families will look like staircases sitting on top of a powerset-like bases. This will allow Reimer's Theorem and Theorem~\ref{sepweight} to give us reasonably tight bounds. Some technicalities arise to make this work for all all possible $(m,n)$.

We call a pair of integers $(n,m)$ \emph{satisfiable} if there exists an $n$-separating union-closed family of size $m$ -- in particular $n$ and $m$ must satisfy $n-1 \leq m \leq 2^n$. Of course for $m=2^n$ the powerset $P_n$ is the only $n$-separating family of the right size. By Theorem~\ref{sepweight} we know already how to construct $n$-separating union-closed families of sizes $m=n-1$ or $m=n$ with minimal weight. Also if $m=n+1$, it is easy to see that the family $T_n\cup \{\emptyset\} \cup\{\{n-1\}\}$ has minimal weight, so for our purposes we may as well assume $2^n>m>n+1$ in what follows.

Given a satisfiable pair $(m,n)$ with $2^n>m>n+1$, there exists a unique integer $b$ such that $2^b-b\leq m-n < 2^{b+1}-(b+1)$. Our aim is to take for our powerset-like base a suitable family of $m-(n-b-1)$ subsets of $[b+1]$, and to place on top of it a staircase of height $n-(b+1)$, thus obtaining a separating union-closed family with the right size and domain.

For such a $b$ we have $2^b+1\leq m-n+b+1 \leq 2^{b+1}$. Write out the binary expansion of $m-n+b+1$ as $2^{b_1}+2^{b_2}+ \ldots 2^{b_t}$ with $0 \leq b_t< b_{t-1} < \ldots <b_1$, and note $b\leq b_1 \leq b+1$. We shall build the base $\mathcal{B}$ of our intermediate family by adding up certain subcubes of $\mathcal{P}[b+1]$.

First of all if $b_1=b+1$, we shall just let $\mathcal{B}$ be the whole of $\mathcal{P}[b+1]$. This is the ``nontechnical case'' of our construction. If on the other hand $b_1=b$, we let $Q_1$ denote the $b_1$-dimensional subcube $\{X \cup\{b+1\}\mid X \subseteq [b] \}$, and for every $i:\ 2\leq i \leq t$ we let 
$Q_i$ be the $b_i$-dimensional subcube $\{X\cup\{b_{i-1} \}\mid X \subseteq [b_i] \}$.  We then set $\mathcal{B}=\bigcup_i Q_i$.

It is easy to see that the $Q_i$ are disjoint. Indeed write $b_0$ for $b+1$ and suppose $i< j$; for every $X \in Q_i$, $b_{i-1}$ is the largest element in $X$ whereas for every $X' \in Q_j$, $b_{j-1}< b_{i-1}$ is the largest element contained in $X'$, so that $X \neq X'$.

\begin{claim}
$\mathcal{B}$ is a $(b+1)$-separating union-closed family. 
\end{claim}
\begin{proof}
$Q_1$ is $(b+1)$-separating since it contains the singleton $\{b+1\}$ and the pairs $\{i, b+1\}$ for every $i< b+1$. Thus $\mathcal{B}$ is $(b+1)$-separating also.

Clearly each of the $Q_i$ is closed under pairwise unions. Now consider $1\leq i < j$ (or alternatively $b_0>b_i>b_j$) and take $X \in Q_i$, $Y\in Q_j$. Then 
\begin{align*}
Y&\subseteq [b_j]\cup \{b_{j-1}\} \\
&\subseteq [b_i],\\
\end{align*}
from which it follows that $X\cup Y \subseteq [b_i]\cup\{b_{i-1}\}$, and hence that $X\cup Y \in Q_i$. Thus $\mathcal{B}=\bigcup_i Q_i$ is closed under pairwise unions, as claimed.
\end{proof}

We now turn to the staircase-like top of our family, $\mathcal{T}$, which we set to be 
\[ \mathcal{T}=\{[b+2], [b+3], \ldots [n] \}.\]

Our intermediate family will then be:
\[\mathcal{S}= \mathcal{B} \cup \mathcal{T}\]
It is easy to see from our construction that $\mathcal{S}$ is union-closed, $n$-separating and has size 
\[ |\mathcal{B}|+|\mathcal{T}|=(m-n+b+1) +(n-b-1)=m.\] 
We do not claim that  $\mathcal{S}$ is an $n$-separating union-closed family of size $m$ with minimal weight;
however as we shall see $w(\mathcal{S})$ is quite close to minimal.

\begin{lemma}\label{techbound}
\[w(\mathcal{B}) < \frac{|\mathcal{B}|\log_2 |\mathcal{B}|}{2}+ |\mathcal{B}|.\]
\end{lemma}
\begin{proof}
In the ``non-technical case'' where $\mathcal{B}=\mathcal{P}[b+1]$ our assertion is trivial.
We turn therefore to the ``technical case'' where $|\mathcal{B}|=2^{b_1}+2^{b_2}+2^{b_3}+ \ldots 2^{b_t}$ with $b=b_1>b_2>\ldots >b_t\geq 0$: 
\begin{align*}
w(\mathcal{B})&= \sum_{i:\ b_i\neq 0} 2^{b_i}\left(\frac{b_i}{2}+1\right)\\
&= \frac{b}2 \sum_{i:\ b_i\neq 0} 2^{b_i} + \sum_{i:\ b_i\neq 0} 2^{b_i}\frac{b_i-b+2}2 \\
&\leq \frac{b|\mathcal{B}|}{2}+2^{b_1}+2^{b_2}/2\\
&< \frac{|\mathcal{B}|\log_2 |\mathcal{B}|}{2} +  |\mathcal{B}|.
\end{align*}

\end{proof}

Now $|\mathcal{B}|\leq m$, and the weight of $\mathcal{T}$ is clearly less than $\frac{n(n+1)}{2}$. Thus it follows that 
\[w(\mathcal{S}) < \frac{m\log_2m}{2} + \frac{n(n+1)}{2}+ m.\]

On the other hand we already know from Reimer's theorem and Theorem~\ref{sepweight} that 
\[w(\mathcal{S}) \geq \max \left(\frac{m\log_2 m}{2},\frac{n(n-1)}{2}\right),\]
which is asymptotically the same except when $n^2\sim m\log_2 m$ when the lower and upper bounds may diverge by a multiplicative factor of at most $2$.

We have thus proved the following theorem.

\begin{theorem} \label{maintheor}
Let $(n,m)$ be a satisfiable pair of integers. Suppose $\mathcal{S}$ is an $n$-separating union-closed family of size $m$ with minimal weight. Then
\[ \max\left(\frac{m\log_2 m}{2}, \frac{n(n-1)}{2}\right) \leq w(\mathcal{S}) \leq \frac{m\log_2m}{2} + \frac{n(n+1)}{2}+ m.\]
In particular if $(n_m, m)_{m\in \mathbb{N}}$ is a sequence of satisfiable pairs and $\mathcal{S}_{m}$ a sequence of $n_m$-separating union-closed families of size $m$ with minimal weight, we have the following:
\begin{itemize}
\item If $n_m/\sqrt{m\log m} \rightarrow 0$ as $m\rightarrow \infty$ then 
\[\lim_{m\rightarrow \infty} w(\mathcal{S}_m)/(\frac{m\log_2 m}{2})=1.\]
\item If $n_m/\sqrt{m\log m}\rightarrow \infty$ as $m\rightarrow \infty$ then 
\[\lim_{m\rightarrow \infty} w(\mathcal{S}_m)/(\frac{n^2}{2})=1.\]
\item Otherwise
\[1\leq \underline{\lim}\ w(\mathcal{S}_m)/\max (\frac{n^2}{2},\frac{m\log_2 m}{2}), \textrm{ and}\]
\[\overline{\lim}\ w(\mathcal{S}_m)/\max (\frac{n^2}{2}, \frac{m\log_2 m}{2})\leq 2\]
\end{itemize}
\qed 
\end{theorem}

As a corollary to Theorems~\ref{sepweight},~\ref{maintheor} and Reimer's Theorem we have the following result regarding average degree.

\begin{corollary}\label{avcol}
Let $\mathcal{S}$ be a separating union-closed family. Then,
\[\frac{1}{|\Omega(\mathcal{S})|}\sum_{x \in \Omega(\mathcal{S})} d_{\mathcal{S}}(x) \geq \frac{\sqrt{|\mathcal{S}|\log_2 |\mathcal{S}|}}{2}+O(1).\]
Moreover there exist arbitrarily large separating union-closed families with
\[\frac{1}{|\Omega(\mathcal{S})|}\sum_{x \in \Omega(\mathcal{S})} d_{\mathcal{S}}(x) \leq \sqrt{|\mathcal{S}|\log_2|\mathcal{S}|}+O(\sqrt{|\mathcal{S}|/\log_2 |\mathcal{S}|}),\]
so our bound is asymptotically sharp except for a multiplicative factor of at most $2$.
\end{corollary}

\begin{proof}
The average degree in a separating family $\mathcal{S}$ is
\[\frac{1}{|\Omega(\mathcal{S})|}\sum_{x \in \Omega(\mathcal{S})} d_{\mathcal{S}}(x)=\frac{w(\mathcal{S})}{|\Omega(\mathcal{S})|}. \] 
If $\mathcal{S}$ is an $n$-separating union-closed family of size $m$, we get two lower bounds on $w(\mathcal{S})$ from Reimer's Theorem and Theorem \ref{sepweight}. Dividing through by $|\Omega(\mathcal{S})|=n$ and optimising yields
\[\frac{1}{|\Omega(\mathcal{S})|}\sum_{x \in \Omega(\mathcal{S})}d_{\mathcal{S}}(x)\geq \frac{\sqrt{|\mathcal{S}|\log_2 |\mathcal{S}|}}{2}-\frac{1}{4}.\]

The constructions from the proof of Theorem \ref{maintheor} then give us for each satisfiable pair $(n,m)$ examples of $n$-separating families of size $m$ with close to minimal average degree. In particular, take $m=2^r$ and $n= \lceil \sqrt{2^rr} \rceil$: the corresponding family we constructed has weight $2^rr+O(2^r)$. It has therefore average degree $\sqrt{r2^r}+O(\sqrt{2^r/r})=\sqrt{m\log_2 m}+O(\sqrt{m/\log_2 m}).$
\end{proof} 
We believe our bounds are in fact asymptotically sharp, and that the constructions we gave in the proof of Theorem~\ref{maintheor} are essentially the best possible. We conjecture to that effect.

\begin{conjecture}\label{strongconje}
Suppose $n=c\sqrt{m\log_2 m}+o(\sqrt{m\log_2 m})$, for some $c>0$, and that $\mathcal{S}$ is an $n$-separating union-closed family of size $m$. Then
\[w(\mathcal{S}) \geq \frac{1+c^2}{2} m\log_2 m +o(m\log_2 m).\]
\end{conjecture}

\section{Minimal $l$-fold weight}

Let $\mathcal{S}$ be a separating union-closed family. Recall that the $l$-fold weight of a family $\mathcal{S}$ is 
\[w_l(\mathcal{S})=\sum_{A\in \mathcal{S}} \binom{|A|}{l}.\]
In the previous section we obtained lower-bounds for $w_1(\mathcal{S})$ in terms of $|\mathcal{S}|$ and $|\Omega(\mathcal{S})|$ and gave constructions showing these were asymptotically sharp up to a multiplicative constant. Using easy generalisations of Reimer's Theorem and Theorem~\ref{sepweight}, we can obtain similar results concerning $w_l(\mathcal{S})$. As a corollary, we will obtain lower bounds on the expected number of sets containing a random $l$-subset of $\Omega(\mathcal{S})$, and show these are again asymptotically sharp up to a constant.

Results in this section are motivated by the remark that repeated iterations of the classical union-closed sets conjecture imply the following stronger looking statement:
\begin{conjecture}[Generalised union-closed sets conjecture]
Let $\mathcal{S}$ be a union-closed family. Then for every integer $l:\ 1\leq l \leq \log_2|\mathcal{S}|$, there is an $l$-subset $X$ of $\Omega(\mathcal{S})$ which is contained in at least $|\mathcal{S}|/2^l$ members of $\mathcal{S}$.
\end{conjecture}

Let us first show how Reimer's Theorem can be immediately generalised to $l$-fold weights.

\begin{lemma} \label{reimerlweight}
Let $l \in \mathbb{N}$ and let $\mathcal{S}$ be a union-closed family. Then

\[w_l(\mathcal{S}) > |\mathcal{S}| \binom{\log_2 |\mathcal{S}|/2}{l}. \]

\end{lemma}
\begin{proof}
The function $x \mapsto \binom{x}{l}$ is convex in $\mathbb{R}^+$. By Jensen's inequality, it follows that
\[w_l(\mathcal{S}) = \sum_{A\in \mathcal{S}} \binom{|A|}{l} \geq |\mathcal{S}| \binom{\sum_{A\in \mathcal{S}} |A|/|\mathcal{S}|}{l}\]
with equality if and only if all the members of $\mathcal{S}$ have the same size. On the other hand, Reimer's average set size theorem tell us
\[\frac{\sum_{A\in \mathcal{S}} |A|}{|\mathcal{S}|} \geq \frac{\log_2 |\mathcal{S}|}{2},\] 
with equality if and only if $\mathcal{S}$ is a powerset (in which case not all the member of $\mathcal{S}$ have the same size). Thus
\[w_l(\mathcal{S}) > |\mathcal{S}| \binom{\log_2 |\mathcal{S}|/2}{l},\]
and this inequality is strict (since we cannot have equality in both Jensen's inequality and Reimer's Theorem.)

Now, the $l$-fold weight of a powerset $P_r =\mathcal{P}([r])$ is
\[w_l(P_r) =\sum_{A:\ |A|=l} \sum_B 1_{A \subseteq B}=2^{r-l} \binom{r}{l}> 2^r \binom{r/2}{l}.\]
However for a fixed $l$,
\[\frac{w_l(P_r)}{2^r\binom{r/2}{l}} \rightarrow 1 \ \textrm{as $r\rightarrow \infty$,} \]
so the  bound on $w_l$ is still asymptotically sharp.
\end{proof}

Next, let us generalise our result that for $\mathcal{S}$ an $n$-separating union-closed family, 
\[w_1(\mathcal{S}) \geq \binom{n}{2}.\]

Again this comes as an easy consequence of Lemmar~\ref{sepstruc}. 

\begin{lemma}\label{seplweight}
Let $l \in \mathbb{N}$ and let $\mathcal{S}$ be a separating union-closed family with $\Omega(\mathcal{S})=[n]$ and elements of $\Omega$ labelled in order of increasing degree $d_{\mathcal{S}}$. Then
\[w_l(\mathcal{S}) \geq \binom{n}{l+1},\]

with equality if and only if $\mathcal{S}$ is of the form 
\[\mathcal{S}=\left\{[n]\setminus[1],[n]\setminus[2], [n]\setminus[3],\ldots [n]\setminus[n-l] \right\}\cup \mathcal{R},\]
where $\mathcal{R}\cup \{[n]\setminus[n-l]\}$ is a separating and union-closed subfamily of $\mathcal{P}([n]\setminus[n-l])$. 
\end{lemma}

\begin{proof}
By Lemma~\ref{sepstruc}, $\mathcal{S}$ contains at least $n-1$ distinct sets $A_i$, $i\in[n-1]$, of the form
\[A_i =\left\{i+1,i+2 \ldots n\right\}\cup X_i, \ X_i \subseteq [i-1].\]
Thus 
\begin{align*}
w_l(\mathcal{S})&\geq \sum_{i \in [n-1]} \binom{|A_i|}{l}\\
& \geq \sum_{i \in [n-1]} \binom{n-i}{l}= \binom{n}{l+1}.\\
\end{align*}

Equality may occur in the above if and only if $A_i=[n]\setminus[i]$ for all $i\leq n-l$ and $\mathcal{S}$ contains no other set of size greater or equal to $l$. Suppose this is the case, and that 
$\mathcal{S}$ contains a set $B$ with $B \cap [n-l]\neq \emptyset$.

Then $B$ contains some $x\in [n-l]$. Suppose it does not contain $n-l+1$. Then by union-closure $B\cup A_{n-l+1}$ is an element of $\mathcal{S}$ of size at least $|\{x,  n-l+2, \ldots n\}|=l$. As it does not contain $n-l+1$, it is not amongst the sets $A_i: i\leq n-l$ we identified earlier, a contradiction. $B$ therefore contains $n-l+1$. By iterating this argument, we see that $B$ must also contain all of $n-l+2, n-l+3, \ldots n-1$. But then $B$ has size at least $|\{x,n-l+1, n-l+2, \ldots n-1\}|=l$. If it does not contain $n$, it is distinct from the sets $A_i: i \leq n-l$ we identified earlier, which is a contradiction. If it does contain $n$, then it has size at least $l+1>l$. This is only possible if $B=A_i$ for some $i \in [n-l]$.

It follows that $\mathcal{S}= \{[n], [n]\setminus\{1\}, [n]\setminus\{2\} \ldots [n] \setminus\{n-l\}\}\cup \mathcal{R}$ with $\mathcal{R}\cup \{[n]\setminus [n-l]\}$ a union-closed and separating subset of $\mathcal{P}([n]\setminus[n-l])$ as required.  
\end{proof}

With Lemmas~\ref{reimerlweight} and~\ref{seplweight} in hand, we can now generalise Theorem~\ref{maintheor}.

\begin{theorem}\label{mainthehorla}
Let $(n,m)$ be a satisfiable pair of integers, and let $l\in \mathbb{N}$. Suppose $\mathcal{S}$ is an $n$-separating union-closed family of size $m$ with minimal $l$-fold weight $w_l(|\mathcal{S}|)=w_l$. Then,
\[\max\left(\binom{n}{l+1},m \binom{\log_2m/2}{l}\right) \leq w_l \]
and
\[w_l\leq \left(\binom{n}{l+1}+m \binom{\log_2m/2}{l}\right)(1+o(1)).\]

\end{theorem}
Again the lower and upper bounds on $w_l$ are asymptotically the same except when $n \sim m^{1/(l+1)} \log_2 m^{1-1/(l+1)}$.

\begin{proof}
As this proof is essentially the same as that of Theorem \ref{maintheor}, we omit the details. The lower bound on $w_l$ follows from Lemmas~\ref{reimerlweight} and~\ref{seplweight}. The upper bound follows from considering the $l$-fold weight of the families we introduced in the proof of Theorem~\ref{maintheor}. The only difficulty involved lies in adapting Lemma~\ref{techbound} to $l$-fold weights. We state and prove below the required generalisation in the ``technical case''.
\begin{lemma}
Let $\mathcal{B}$ be as defined in the previous section, and assume $|\mathcal{B}|=2^b+2^{b_2}+\ldots 2^{b_t}$. Then
\[w_l(\mathcal{B})< \left(1+\frac{2l}{\log_2 |\mathcal{B}|}\right)\frac{|\mathcal{B}|}{l!}\left( \frac{\log_2 |\mathcal{B}|}{2}\right)^l.\]
\end{lemma}
\begin{proof}
\begin{align*}
w_l(\mathcal{B})&= \sum_i \binom{b_i}{l}2^{b_i-l}+ \binom{b_i}{l-1}2^{b_i-l+1}\\
&\leq \left( \binom{b}{l}+ 2\binom{b}{l-1}\right)\sum_i 2^{b_i-l}\\
&< (1+\frac{2l}{b}) \frac{b^l}{l!}|\mathcal{B}|\\
&< \left(1+\frac{2l}{\log_2 |\mathcal{B}|}\right)\frac{|\mathcal{B}|}{l!}\left( \frac{\log_2 |\mathcal{B}|}{2}\right)^l.
\end{align*}
\end{proof}
Theorem~\ref{mainthehorla} follows straightforwardly from here.
\end{proof}

As in the previous section we can use our result on $l$-fold weights to obtain information about the average number of sets containing a randomly chosen $l$-subset in a separating union-closed family. 

\begin{corollary}
Let $\mathcal{S}$ be a separating union-closed family, and let $X$ be an $l$-subset of $\Omega(\mathcal{S})$ chosen uniformly at random. Then
\[\mathbb{E}_Xd_{\mathcal{S}}(X) \geq {|\mathcal{S}|}^{\frac{1}{l+1}} {\left(\frac{\log_2 |\mathcal{S}|}{2(l+1)}\right)}^{1-\frac{1}{l+1}}+O\left({\left(\frac{|\mathcal{S}|}{\log_2|\mathcal{S}|}\right)}^{\frac{1}{l+1}}\right). \]
Moreover there exist arbitrarily large separating union-closed families $\mathcal{S}$ with
\[\mathbb{E}_Xd_{\mathcal{S}}(X) \leq 2{|\mathcal{S}|}^{\frac{1}{l+1}} {\left(\frac{\log_2 |\mathcal{S}|}{2(l+1)}\right)}^{1-\frac{1}{l+1}}+O\left({\left(\frac{|\mathcal{S}|}{\log_2|\mathcal{S}|}\right)}^{\frac{1}{l+1}}\right), \]
so this bound is asymptotically sharp except for a multiplicative factor of at most $2$.
\end{corollary}
\begin{proof}This is instant from Lemma~\ref{reimerlweight}, Lemma~\ref{seplweight} and Theorem \ref{mainthehorla}.
\end{proof}

We end our paper with the natural generalisation of Conjecture~\ref{strongconje}.

\begin{conjecture}
Let $l$ be an integer. Suppose $n=n(m)$ satisfies
\[n=cm^{1/l+1} \left(\log_2 m\right)^{1-1/(l+1)}  (1+o(1))\] 
for some $c=c(m)$. Then if $\mathcal{S}$ is an $n$-separating union-closed family of size $m$, its $l$-fold weight satisfies

\[w_l(\mathcal{S}) \geq m (\log_2 m)^l \left(\frac{1}{l!2^l}+ \frac{c^{l+1}}{(l+1)!}\right)(1+o(1)).\]

\end{conjecture}


\begin{thebibliography}{99}

\bibitem{abenakano}
 T. Abe and B. Nakano, \textit{Frankl's conjecture is true for modular lattices}, Graphs and Combinatorics \textbf{14} (1998), 305-311. 

\bibitem{bosnmarkov}
I. Bos\v{n}jak and P. Markov\'ic, \textit{The 11-element case of Frankl's conjecture}, Electronic Journal of Combinatorics \textbf{15}, (1): R88.

\bibitem{czedli} 
G. Cz\'edli, \textit{On averaging Frankl's conjecture for large union-closed sets}, Journal of Combinatorial Theory - Series A \textbf{116} (2009), 24-729.

\bibitem{czedlietal1} 
G. Cz\'edli, M. Mar\'oti and E. T. Schmidt, \textit{On the scope of averaging for Frankl's conjecture}, Order \textbf{26} (2009), 31-48.

\bibitem{czedlietal2} 
G. Cz\'edli and E. T. Schmidt, \textit{Frankl's conjecture for large semimodular and planar semimodular lattices}, Acta Univ. Palacki. Olomuc., Fac. rer. nat., Mathematica \textbf{47} (2008), 47-53.

\bibitem{frankl} 
P. Frankl, \textit{Extremal set systems. Handbook of combinatorics}, Vols. \textbf{1, 2},  1293-1329, Elsevier, Amsterdam, 1995.

\bibitem{knill} 
E. Knill, \textit{Graph generated union-closed families of set}, (1993), unpublished manuscript.

\bibitem{morris} 
R. Morris, \textit{FC-families and improved bounds for Frankl's conjecture}, European Journal of Combinatorics \textbf{27} (2006), 269-282. 

\bibitem{lofaro} 
G. Lo Faro, \textit{Union-closed sets conjectures: improved bounds}, Journal of Combinatorial Mathematics and Combinatorial Computing \textbf{16}, 97-102.

\bibitem{poonen} 
B. Poonen, \textit{Union-closed families}, Journal of Combinatorial Theory --- Series A \textbf{59} (1992), 253-268.

\bibitem{reimer}
D. Reimer, \textit{An Average Set Size Theorem}, Combinatorics, Probability and Computing \textbf{12} (2003), 89-93.

\bibitem{roberts} 
I. Roberts, \textit{The union closed sets conjecture}, Technical Report No 2/92, School of Mathematical Statistics, Curtin University of Technology, Perth (1992).

\bibitem{stanley} 
R. P. Stanley, \textit{Enumerative Combinatorics}, Vol. \textbf{1}, Wadsworth and Brooks/Coole, Belmont CA, 1996. 

\bibitem{wojcik1}
 P. W\'ojcik, \textit{Union-closed families of sets}, Discrete Mathematics \textbf{199} (1999), 173-182.

\bibitem{wojcik2} 
P. W\'ojcik \textit{Density of union-closed families}, Discrete Mathematics \textbf{105} (1992), 259-267.

\end{thebibliography}
\end{document}